\documentclass[12pt]{amsart}

\usepackage{latexsym, amssymb, amsmath,ulem,soul,esint}
\usepackage{tikz}
\usepackage{hyperref}
\usepackage{amsfonts, graphicx}
\usepackage{graphicx,color}
\newcommand{\be}{\begin{equation}}
\newcommand{\ee}{\end{equation}}
\newcommand{\beq}{\begin{eqnarray}}
\newcommand{\eeq}{\end{eqnarray}}

\usepackage{wasysym,stmaryrd}
\theoremstyle{definition}

\newtheorem{df}{Definition}
\newtheorem{thm}{Theorem}
\newtheorem{cl}{Corollary}

\newtheorem{lm}{Lemma}

\newtheorem{rmk}{Remark}

\usepackage{mathrsfs} 

\theoremstyle{remark}
\numberwithin{equation}{section}

\def\be{\begin{equation}}
\def\ee{\end{equation}}

\newcommand{\R}{\mathbb{R}}

\newcommand{\M}{\mathcal{M}^{n}}
\newcommand{\N}{\mathcal{N}}

\newcommand{\smf}{\mathcal{C}^{\infty}(\M)}

\DeclareMathOperator{\Ric}{Ric}
\DeclareMathOperator{\Hess}{Hess}
\DeclareMathOperator{\vol}{Vol}
\DeclareMathOperator{\BigO}{O}

\DeclareMathOperator{\Rm}{Rm}

\DeclareMathOperator{\scal}{scal}


\title[]
{Steady soliton with $\mathcal{L}^{1}$ decay curvature}

\author{Ming Hsiao}
\address[Ming Hsiao]{Department of Mathematics
National Taiwan University
and
National Center for Theoretical Sciences,
Math Division,
Taipei 10617,
Taiwan
}
\email{minghsiao@ncts.ntu.edu.tw}

\subjclass[2020]{53C25, 53E20}

\date{\today}
\begin{document}

\begin{abstract}
In this paper, we establish a compactness theorem for gradient Ricci solitons with scalar curvature bounds and uniform lower bounds of harmonic coordinates. Our approach is to bootstrap regularity in harmonic coordinates by exploiting the soliton equation. As an application, we show that the regular part of any noncollapsed limit of gradient Ricci solitons with bounded Ricci curvature is smooth. Further, we show that a steady gradient Ricci soliton is asymptotically cylindrical under an $\mathcal{L}^{1}$-decay assumption on its Ricci curvature.
\end{abstract}

\keywords{compactness theorem, Ricci soliton, harmonic coordinate}

\maketitle

\markboth{Ming Hsiao}{Steady soliton with $\mathcal{L}^{1}$ decay curvature}

\section{Introduction}

Given a Riemannian manifold $(\M,g)$ equipped with a smooth function $f\in\smf$, the triple $(\M,g,f)$ is called a \textit{gradient shrinking/steady/expanding Ricci soliton} if 
\begin{equation}
    \Ric_{g}+\Hess f=\lambda g,
\end{equation}
holds on $\M$ for $\lambda=\frac{1}{2},0,-\frac{1}{2}$, respectively. Since the introduction of the Ricci flow by Hamilton \cite{H82}, Ricci solitons have been studied extensively since it arises as a self-similar solution of Ricci flow and as a canonical model of singularity information. In particular, they play a crucial role in Perelman's resolution of the geometrization conjecture on three-manifolds \cite{H95, P02, B20, Ba20b}. From a Riemannian geometric viewpoint, the gradient Ricci soliton can be viewed as a generalization of an Einstein manifold, that is, the Bakry-Émery tensor $\Ric_{f}:=\Ric+\Hess(f)=\lambda g$ on $\M$.

\vskip0.2cm
Compactness theorems for Ricci solitons serve as a fundamental tool in understanding their geometric structure. In the shrinking case, Zhang established a smooth compactness theorem (with orbifold limit) for the family of shrinking solitons with uniform bounds on diameter, lower bound on $\Ric$ and volume, and an upper bound on energy. Subsequently, Haslhofer and Müller removed the first three conditions, replacing them with a uniform lower bound on Perelman's $\mu$-entropy \cite{Z06, HM11}. In a series of breakthrough works, Bamler introduces the novel notion, \textit{$\mathbb{F}$-convergence}, of the convergence of Ricci flow \cite{Ba20a, Ba20b, Ba23}. Using this framework, Li and Wang proved the compactness theorem of shrinking Ricci soliton with uniform entropy lower bound and analyzed the singular strata in the spirit of Cheeger-Colding theory \cite{LW24}. In contrast, for expanding and steady Ricci solitons, a general compactness theory remains largely undeveloped. Chen obtained a compactness theorem for gradient Ricci solitons based on Ricci curvature bounds and a lower bound of injectivity radius \cite{C20}. In the case of the asymptotically conical expander, Deruelle established a compactness theorem with nonnegative Ricci curvature and uniform decay of $\nabla^{k}\Rm$ \cite{D17}.

\vskip0.2cm

Our first result concerns a compactness theorem with a uniform lower bound of harmonic radius (see Definition \ref{df-harmonic radius}) and a uniform bound of scalar curvature. More precisely, we consider the following class:

\begin{df}
    Let $K,r_{0}:[0,\infty)\rightarrow(0,\infty)$ be two positive continuous functions. The class $\mathfrak{M}(n,K,A,r_{0},\lambda,\alpha,\Lambda,C_{0})$ consists the complete pointed gradient Ricci soliton $(\M,g,f,p)$, which satisfies $\Ric+\Hess(f)=\lambda g$,  with the following conditions:
\begin{equation}
    |{\scal}|(q)\leq K(d_{g}(p,q)),\ r_{H}(\Lambda,1,\alpha)(q)\geq r_{0}(d_{g}(p,q)),\text{ and }|f|(p)\leq A,
\end{equation}
for all $q\in\M$ and 
\begin{equation}\label{eq-soliton identity with scal and nabla f}
    \scal+|{\nabla f}|^{2}=2\lambda f+C_{0}.
\end{equation}
\end{df}

\vskip0.2cm
Note that the condition \eqref{eq-soliton identity with scal and nabla f} holds automatically for any gradient Ricci soliton. Indeed, we have the identity
\begin{equation}
    \nabla\left(\scal+|{\nabla f}|^{2}-2\lambda f\right)=0,
\end{equation}
see \cite{CH24} for instance. We prove that the class $\mathfrak{M}$ is compact under smooth topology.
\begin{thm}\label{thm-compactness of Ricci soliton}
    $\mathfrak{M}(n,K,A,r_{0},\lambda,\alpha,\Lambda,C_{0})$ is compact under smooth topology.
\end{thm}
    When $K$ and $r_{0}$ are fixed constants, this result can be viewed as a generalization of \cite[Corollary 1.4]{C20}. Indeed, a two-sided Ricci curvature bound together with a uniform lower bound on the injectivity radius yields uniform $\mathcal{C}^{1,\alpha}_{H}$ bounds.

\vskip0.2cm
In the case of the steady soliton, we established the uniform curvature bound when $r_{0}$ is a constant.
\begin{cl}\label{cl-bounded curvature under uniform lower bound of r(1,alpha)}
    Let $(\M,g,f)$ be a complete gradient steady Ricci soliton. Suppose that $\inf_{\M} r_{H}(\Lambda,1,\alpha)>0$ for some $\Lambda>1$ and $\alpha\in(0,1)$. Then $\Rm$ is bounded on $\M$.
\end{cl}

Another application of the study on harmonic coordinate is to reinforce the regularity of the non-collapsing two-sided Ricci limit space of gradient Ricci solitons, which is motivated by the case of Einstein manifolds \cite{CN15}.
\begin{cl}\label{cl-GH limit of gradient Ricci soliton}
    Let $(\M_{i},g_{i},f_{i},p_{i})$ be a sequence of pointed complete gradient Ricci solitons and  $\lambda,v,A>0$ be the constants. Assume
    \begin{enumerate}
        \item $\Ric_{i}+\Hess f_{i}=\lambda_{i}g_{i}$ with constant $|\lambda_{i}|\leq\lambda$ for all $i$.
        \item $|{\Ric_{i}}|\leq1$. for all $i$.
        \item $\vol_{i}(p_{i},1)\geq v$ for all $i$.
        \item $|f_{i}|(p_{i})+|{\nabla f_{i}}|(p_{i})\leq A$ for all $i$.
    \end{enumerate}
    Then, after passing to a subsequence, there is a complete pointed length space $(X_{\infty},d_{\infty},p_{\infty})$ and a locally Lipschitz function $f_{\infty}:X_{\infty}\rightarrow\R$ such that
    \begin{equation}
        (\M_{i},d_{g_{i}},f_{i},p_{i})\overset{\text{pGH}}{\longrightarrow}(X_{\infty},d_{\infty},f_{\infty},p_{\infty}),
    \end{equation}
    as $i\rightarrow\infty$. Moreover, there is a decomposition $X_{\infty}=\mathcal{R}\sqcup\mathcal{S}$ with the following:
    \begin{enumerate}
        \item There is a smooth Riemannian $n$-manifold structure $(\mathcal{R},g_{\infty},f_{\infty}|_{\mathcal{R}})$ such that $(\mathcal{R},d_{g_{\infty}})$ and $(\mathcal{R},d_{\infty}|_{\mathcal{R}})$ are local isometry and there exists a constant $\lambda_{\infty}$ with $|{\lambda_{\infty}}|\leq\lambda$ such that 
        \begin{equation}
            \Ric_{g_{\infty}}+\Hess_{g_{\infty}}f_{\infty}=\lambda_{\infty}g_{\infty},
        \end{equation}
        on $\mathcal{R}$.
        \item $\dim_{\mathcal{H}}\mathcal{\mathcal{S}}\leq n-4$.
    \end{enumerate}
\end{cl}

\vskip0.2cm
Curvature estimates for steady Ricci solitons remain a central open problem in the analysis of Ricci solitons. A folklore conjecture asserts complete gradient steady Ricci soliton has bounded curvature. In \cite{CH25}, we confirm this conjecture in the case that $|{\Ric}|=\BigO(r^{-1-\epsilon})$ and proper potential function $f$; and $\Rm$ growth at most polynomially when $|{\Ric}|=\BigO(r^{-1})$. If one only assumes that $\Ric$ is bounded and nonnegative, Wu proved that $\Rm$ grows at most exponentially. 

\vskip0.2cm
An application of Corollary \ref{cl-bounded curvature under uniform lower bound of r(1,alpha)} is the asymptotic structure of steady Ricci solitons with $\mathcal{L}^{1}$-decay $\Ric$. Given a steady soliton $(\M,g,f)$, we make the following assumptions:
\begin{itemize}
\makeatletter
    \item [\textbf{(A0)}]\def\@currentlabel{\textbf{(A0)}}\label{Ass-0}$(\M,g)$ is non-Ricci flat open manifold and $f$ is a proper and negative function such that
    \begin{equation}
        \scal+|{\nabla f}|^{2}=1.
    \end{equation}
\makeatother
\makeatletter
    \item [\textbf{(A1)}]\def\@currentlabel{\textbf{(A1)}}\label{Ass-1}$|{\Ric}|(x)=\BigO(h(r(x)))$ as $r(x):=d_{g}(p,x)\rightarrow\infty$ for some $h(r)\in\mathcal{L}^{1}([1,\infty))$ and $h(r)$ is non-increasing for $r\gg1$.
\makeatother
\end{itemize}
\vskip 0.2cm

 For steady soliton with \ref{Ass-0} and \ref{Ass-1}, we introduce the following notion. Denote $F:=-f$, the level set $\Sigma_{R}=\{F=R\}$, and the flow $\frac{d}{dt}\varphi_{t}=\frac{\nabla F}{|{\nabla F}|^{2}}$ on the region $\{F\geq R_{0}\}\subset\{|{\nabla F}|>\frac{1}{2}\}$. Define the map $\Psi:\{F\geq R_{0}\}\rightarrow\Sigma_{R_{0}}\times(R_{0},\infty)$ by $\Psi(x):=(\varphi_{R_{0}-F(x)}(x),F(x))$. The following theorem indicates the asymptotic geometry of $\M$ is $\R\times\N$ for some compact Ricci flat manifold $\N$.
 
\begin{thm}\label{thm-bounded Rm with L1 condition}
    Let $(\M,g,f)$ be a gradient steady Ricci soliton. Suppose $(\M,g,f)$ satisfying \ref{Ass-0} and \ref{Ass-1}. Then $\Rm$ is bounded. Furthermore, there exists closed Ricci flat manifold $(\N^{n-1}=\Sigma_{R_{0}},g_{\infty})$ such that $(\M,g,p_{k})$ converge \textit{smoothly} to $(\R\times\N,\overline{g_{\infty}}:=dr^{2}+g_{\infty},\bar{o})$ whenever $\varphi_{R_{0}-F(p_{k})}(p_{k})$ converge to $o\in\N$. Moreover, for any $k\in\mathbb{N}_{0}$, there is a constant $C_{k}>0$ such that 
    \begin{equation}
        \|(\Psi^{-1})^{*}g-\overline{g_{\infty}}\|_{\mathcal{C}^{k}(\overline{g_{\infty}})}\leq C_{k}\int_{r(x)}^{\infty}h(s)ds,
    \end{equation}
    for all $r(x)\gg1$.
\end{thm}

\vskip 0.2cm 
A particularly important case arises when the Ricci curvature decays exponentially. Most known examples, such as \textit{the product of the cigar soliton and a closed Ricci flat manifold} or \textit{$\R\times \text{closed Ricci flat}$}, exhibit exponential decay of the Ricci curvature. As a corollary of Theorem \ref{thm-bounded Rm with L1 condition}, we establish the optimal convergence rate of steady Ricci solitons with exponential decay of the Ricci curvature.
\begin{cl}\label{cl-curvature estimate with h(r)=exp(-r)}
     Let $(\M,g,f)$ be a gradient steady Ricci soliton. Suppose $(\M,g,f)$ satisfying \ref{Ass-0} and \ref{Ass-1} with $h(r)=e^{-\kappa r}$ for some constant $\kappa>0$.  Then there exists closed Ricci flat manifold $(\N^{n-1}=\Sigma_{R_{0}},g_{\infty})$ such that for any $k\in\mathbb{N}_{0}$, there is a constant $C_{k}>0$ such that 
     \begin{equation}
         \|\Psi^{*}\overline{g_{\infty}}-g\|_{\mathcal{C}^{k}(g)}\leq C_{k}e^{-\kappa r(x)},
     \end{equation}
     for all $r(x)\gg1$.
\end{cl}

\vskip0.2cm
Combining this with known curvature estimates for steady Ricci solitons \cite{MSW19,C19}, we derive the following curvature estimate under a stronger decay assumption on $h$, namely that the integral of $h$ satisfies a Dini-type condition.
\begin{cl}\label{cl-curvature estimate with Dini H}
    Let $(\M,g,f)$ be a gradient steady Ricci soliton. Suppose $(\M,g,f)$ satisfying \ref{Ass-0} and \ref{Ass-1} with $\frac{H(r)}{r}:=\int_{r}^{\infty}h(t)dt\in\mathcal{L}^{1}([1,\infty))$ and $\lim_{r\rightarrow\infty}H(r)=0$. Suppose that $\lim_{r\rightarrow\infty}|{\Rm}|=0$. Then there exists closed flat manifold $(\N^{n-1}=\Sigma_{R_{0}},g_{\infty})$ such that for any $k\in\mathbb{N}_{0}$, there is a constant $C_{k}>0$ such that 
     \begin{equation}
         \|\Psi^{*}\overline{g_{\infty}}-g\|_{\mathcal{C}^{k}(g)}\leq C_{k}e^{-r(x)},
     \end{equation}
     for all $r(x)\gg1$. Moreover, $|{\Rm}|=\BigO(e^{-r(x)})$ holds as $r(x)\rightarrow\infty$.  
\end{cl}

\vskip 0.2cm
The structure of this paper is as follows: We recall some elementary properties of the gradient steady soliton in Section \ref{preliminary}. In section \ref{compactness}, we prove the compactness theorem of steady solitons and discuss some applications. Finally, we confirm the convergence rate of steady Ricci solitons with $\mathcal{L}^{1}$ decay of Ricci curvature and its applications.

\vskip0.2cm
{\it Acknowledgement}: The author thanks Pak-Yeung Chan for helpful discussions. 

\section{Preliminary of steady Ricci soliton}\label{preliminary}
In this section, we recall some basic properties of steady Ricci solitons that we will use in later sections.
\begin{lm}\label{seqn}\cite{H95}
    Let $(\M,g,f)$ be a complete gradient steady Ricci soliton. Then
    \begin{align}
        &\label{eq-scal+Delta f=0}\scal+\Delta f=0;\\
        &\nabla_{i}\scal=2\Ric_{ij}\nabla_{j}f;\\
        &\label{normalized cond}\scal+|\nabla f|^{2}=C_{0},
    \end{align}
    for some constant $C_{0}$.
\end{lm}
$(\M,g,f)$ is said to be normalized gradient steady soliton if $C_{0}=1$. 
\begin{rmk}
    As a result of Chen \cite{C09}, the scalar curvature $\scal\geq0$ for a complete steady soliton. Therefore, $C_{0}=0$ only if $f$ is constant, which implies $(\M,g)$ is Ricci flat. Hence, for a non-Ricci flat complete steady soliton, upon scaling the metric, we may as well assume it is normalized such that $C_0=1$.  
\end{rmk}

\begin{lm}\label{lm-linear equivalent of F}
    Let $(\M,g,f)$ be a steady Ricci soliton with \ref{Ass-0} and \ref{Ass-1}. Then there is a $C>1$ such that
    \begin{equation}
        r-C\leq F:=-f\leq r-f(p), 
    \end{equation}
    for $r(x):=d_{g}(p,x)$ for some $p\in\M$.
\end{lm}
\begin{proof}[Proof of Lemma \ref{lm-linear equivalent of F}]
    We adopt the approach in  \cite[Lemma 2]{C19} and \cite[Corollary 3.4]{CMZ22}. Since $(\M,g,f)$ is a complete gradient Ricci soliton, $\scal\geq0$ holds on $\M$ by \cite{C09}. Therefore, $|{\nabla F}|\leq1$ holds on $\M$. This completes the proof of the right-hand side. For the left-hand side, since $\scal\rightarrow0$ as $r\rightarrow\infty$ and $F$ is proper, we may assume $|{\nabla F}|\geq\frac{1}{2} $ whenever $F\geq R_{0}$. Denote the flow $\varphi_{t}$ with $\frac{d}{dt}\varphi_{t}=\frac{\nabla F}{|{\nabla F}|^{2}}$ with $\varphi_{0}=\text{id}$ on $F\geq R_{0}$. Then for $F(x)=R>R_{0}$, 
    \begin{equation}
        \begin{aligned}
            d_{g}(x,\varphi_{R_{0}-R}(x))\leq&~\int_{0}^{R-R_{0}}\frac{1}{|{\nabla F}|}(\varphi_{R_{0}-R+t}(x))dt\\
            =&~\int_{0}^{R-R_{0}}\frac{1}{{\sqrt{1-\scal}}}(\varphi_{R_{0}-R+t}(x))dt\\
            \leq&\int_{0}^{R-R_{0}}1+\frac{1}{2}\scal(\varphi_{R_{0}-R+t}(x))dt\\
            \leq&~R-R_{0}+C\int_{R_{0}}^{\infty}h(t)dt\\
            \leq&~F(x)+C.
        \end{aligned}
    \end{equation}
    Also, $d_{g}(x,\varphi_{R_{0}-R}(x))\geq r(x)-\sup_{q\in\{F=R_{0}\}}d_{g}(p,q)=:r(x)-C_{0}$ whenever $F(x)>R_{0}$. Since $F$ is proper, this completes the proof. 
\end{proof}
We recall a general result of the end at infinity of complete gradient steady Ricci solitons by \cite{MW11}.
\begin{lm}\cite{MW11}
    Let $(\M,g,f)$ be a non-Ricci flat complete gradient steady Ricci soliton. Then it has at most one end. 
\end{lm}
\section{A compactness theorem of gradient Ricci soliton}\label{compactness}
We recall the definition of harmonic radius.
\begin{df}\label{df-harmonic radius}[Harmonic radius, \cite{H19}] Given a Riemannian manifold $(\M,g)$ and a point $x\in\M$. For $\Lambda\geq1$, $k\in\mathbb{N}_{0}$, and $\alpha\in(0,1)$, the number $r_{H}(\Lambda,k,\alpha)(x)$ denotes the largest radius of the ball $B_{g}(x,r_{H})$, such that there exists a harmonic coordinate such that
\begin{equation}
    \Lambda^{-1}\delta_{ij}\leq g_{ij}\leq \Lambda\delta_{ij},
\end{equation}
and 
\begin{equation}
    \sum_{1\leq|\beta|\leq k}r_{H}^{|\beta|}\sup|\partial^{\beta}g_{ij}|+\sum_{|\beta|= k}r_{H}^{k+\alpha}\sup_{z\neq y\in B_{g}(x,r_{H})}\frac{|\partial^{\beta}g_{ij}(z)-\partial^{\beta}g_{ij}(y)|}{d_{g}(z,y)^{\alpha}}\leq \Lambda-1.
\end{equation}
    For $p>n$, the number $r_{H}(\Lambda,k,p)(x)$ denotes the largest radius of the ball $B_{g}(x,r_{H})$, such that there exists a harmonic coordinate such that
\begin{equation}
    \Lambda^{-1}\delta_{ij}\leq g_{ij}\leq \Lambda\delta_{ij},
\end{equation}
and 
\begin{equation}
    \sum_{1\leq|\beta|\leq k}r_{H}^{2-\frac{n}{p}}\|\partial^{\beta}g_{ij}\|_{\mathcal{L}^{p}}\leq \Lambda-1.
\end{equation}
Also, we consider the notations $\mathcal{C}^{k,\alpha}_{H}(x,r)$ and $\mathcal{W}^{k,p}_{H}(x,r)$ as the infimum of $\Lambda$ over all harmonic coordinate on ball $B_{g}(x,r)$.
\end{df}
\vskip0.2cm
One important identity under the harmonic coordinate is 
\begin{equation}\label{eq-Bochner formula under harmonic coordinate}
    \Delta g_{ij}=-2\Ric_{ij}+Q(g,\partial g)_{ij}.
\end{equation}
Therefore, a standard strategy in harmonic coordinates is to use the Schauder estimate to bootstrap the regularity of $g$. 
Since the regularity of $\Ric$ is equivalent to the regularity of $\Hess f$ on Ricci soliton, we are able to use the soliton identity to establish the $\mathcal{C}^{\alpha}_{H}$ bound of $\Ric$ when $(\M,g,f,p)\in\mathfrak{M}(n,K,A,r_{0},\lambda,\alpha,\Lambda,C_{0})$, which implies the local uniform curvature bounds by the Schauder estimate.
\begin{proof}[Proof of Theorem \ref{thm-compactness of Ricci soliton}]
We begin with the following a priori estimate of $\Rm$.
\begin{lm}\label{lm-uniform curvature bound in class m}
    Let $(\M,g,f,p)\in\mathfrak{M}(n,K,A,r_{0},\lambda,\alpha,\Lambda,C_{0})$. Then for any $R>0$ and $k\geq0$, there is a constant $C_{R,k}(n,K,A,r_{0},\lambda,\alpha,\Lambda,C_{0})>0$ such that $|{\nabla^{k}\Rm}|_{g}\leq C_{R,k}$ on $B_{g}(p,R)$. Furthermore, in the case of $K,r_{0}$ are constant, if $|{\nabla f}|\leq A$ holds on $\M$ or $\lambda=0$, then $C_{R,k}$ is independent of $R$ .
\end{lm}
\begin{proof}[Proof of Lemma \ref{lm-uniform curvature bound in class m}]
We first consider the case $\lambda\neq0$. Thanks to the compactness theorem, it suffices to show that for any $R>0$ and $k\in\mathbb{N}_{0}$, there is a uniform upper bound $C_{R,k}$ such that $|{\nabla^{k}\Rm}|\leq C_{R,k}$ on $B_{g}(p,R)$ whenever $(\M,g,f,p)\in\mathfrak{M}(n,K,A,r_{0},\lambda,\alpha,\Lambda,C_{0})$. Note that since $r_{H}(\Lambda,1,\alpha)(p)\geq r_{0}(0)$, the injectivity radius at $p$ follows from the local $\Rm$ bound on $B_{g}(p,r_{0}(0))$ by the classical Cheeger-Gromov-Taylor theorem. 

\vskip0.2cm
Fixed $R>0$ and $q\in\M$ with $q\in B_{g}(p,R)$. Define $r_{0,R}:=\inf_{B_{g}(p,R)}r_{0}$ and $K_{R}:=\sup_{B_{g}(p,R)}K$. Since $r_{H}(\Lambda,1,\alpha)(q)\geq r_{0,R}$, there is a harmonic coordinate chart on $B_{g}(q,r_{0,R})$ with uniform $\mathcal{C}^{1,\alpha}_{H}(q,r_{0,R})$ bound. By (\ref{eq-soliton identity with scal and nabla f}) and $|{\scal}|\leq nK_{R}$, 
\begin{equation}
    |{\nabla\sqrt{2\lambda f+(C_{0}+nK_{R})}}|\leq|{\lambda}|,
\end{equation}
whenever $\lambda\neq0$, Hence, there is a constant $C$, only depending on $n,K_{R},A,r_{0,R},\lambda,\alpha,\Lambda,C_{0}$, such that $|{f(q)}|\leq CR^{2}+C$. Combining with (\ref{eq-soliton identity with scal and nabla f}), we obtain a uniform  $\mathcal{L}^{\infty}$ bound for $f$ and $|{\nabla f}|$ (may depend on $R$) on the harmonic coordinate ball. 

\vskip0.2cm
Note that our goal is to bootstrap the regularity to $\mathcal{C}^{2,\alpha}_{H}(q,cr_{0,R})\leq \Lambda'$ for some universal $c\in(0,1)$ and $\Lambda'$ may depend on all the constant above. This provides a uniform upper bound of $\Rm$ on $B_{g}(q,cr_{0,R})$, and the higher derivative of $\Rm$ follows from local Shi's estimate. From now on, the constant $C$, which only depends on $n,K_{R},A,r_{0,R},\lambda,\alpha,\Lambda,C_{0},R$, may change from line to line.

\vskip0.2cm
Taking the trace of the soliton equation gives $\scal+\Delta f=n\lambda$. By the uniformly ellipticity, uniform $\mathcal{C}^{1,\alpha}_{H}(q,r_{0,R})$, and the $\mathcal{L}^{p}$-theory, we obtain a uniform $\mathcal{W}^{2,p}$ bound of $f$ on $B_{g}(q,\frac{r_{0,R}}{2})$ for $p>n/2$ since $\scal$ and $|{f}|+|{\nabla f}|$ are bounded by $C$. Hence, the Sobolev inequality implies a uniform $\mathcal{C}^{1,\alpha}$ bound of $f$ on $B_{g}(q,\frac{r_{0,R}}{2})$. Combining with (\ref{eq-soliton identity with scal and nabla f}), this implies a $\mathcal{C}^{\alpha}$ bounds of $\scal$ on $B_{g}(q,\frac{r_{0,R}}{2})$.

 \vskip0.2cm
 Also, since $\scal+\Delta f=n\lambda$, the Schauder estimate (which is uniform by the $\mathcal{C}^{1,\alpha}_{H}(q,\frac{r_{0,R}}{2})\leq\Lambda$), we yield the uniform $\mathcal{C}^{2,\alpha}$ bounds for $f$ on $B_{g}(q,\frac{r_{0,R}}{4})$. 

 \vskip0.2cm
 Applying the soliton equation again, we deduce the uniform $\mathcal{C}^{\alpha}$ bounds for $\Ric$ on $B_{g}(q,\frac{r_{0,R}}{4})$. Finally, in harmonic coordinate, this implies the uniform $\mathcal{C}^{2,\alpha}_{H}(q,\frac{r_{0,R}}{8})$ bounds of metric $g$ by Schauder estimate, and hence a uniform bound of $\Rm$ on $B_{g}(q,\frac{r_{0,R}}{8})$. See \cite{A90} for instance.

\vskip0.2cm
For the case of $\lambda=0$ and $r_{0},K$ are constants, note that $C_{0}$ is bounded from above by $K,A$ by (\ref{eq-soliton identity with scal and nabla f}). Therefore,
\begin{equation}
    |{\nabla f}|^{2}\leq C_{0}+nK,
\end{equation}
on $\M$. If the $|{\nabla f}|\leq C$ for some $C>0$ on $\M$, for each $q\in\M$, we may assume $f(q)=0$ by adjusting $f$ by $f-f(q)$. Then all the estimates above are independent of $R$, which completes the proof.
\end{proof}

The precompactness property then follows from Lemma \ref{lm-uniform curvature bound in class m}, Shi's estimate \cite[Remark 2.7]{D17}, and uniform $\mathcal{C}^{1,\alpha}_{h}$ bounds at $p$.
\end{proof}

\begin{proof}[Proof of Corollary \ref{cl-bounded curvature under uniform lower bound of r(1,alpha)}]
    Since $(\M,g,f)$ is a complete steady soliton, we may assume $\scal\geq0$ and $\scal+|{\nabla f}|^{2}=1$. Also, for any $p\in\M$, we may assume $f(p)=0$ by replacing $f$ with $f-f(p)$. Then by assumption, 
    \begin{equation}
        (\M,g,f,p)\in\mathfrak{M}\left(n,1,1,\frac{1}{2}\inf_{\M}r_{H}(\Lambda,1,\alpha),\alpha,\Lambda,1\right),
    \end{equation}
    and $|{\nabla f}|\leq1$ on $\M$. By Lemma \ref{lm-uniform curvature bound in class m}, this implies a uniform upper bound $\Rm$.
\end{proof}
\begin{proof}[Proof of Corollary \ref{cl-GH limit of gradient Ricci soliton}]
    The existence of $(X_{\infty},d_{\infty},p_{\infty})$, $\mathcal{C}^{1,\alpha}$-structure on $(\mathcal{R},g_{\infty})$, and the decomposition $X=\mathcal{R}\sqcup\mathcal{S}$ with $\dim_{\mathcal{H}}S\leq n-4$ follow from \cite{CN15, A90} since we assume two-sided Ricci bound and lower bound of volume. It suffices to show the $\mathcal{C}^{\infty}_{\text{loc}}$ convergence on $\mathcal{R}$. Since for each $\M_{i}$, there is a constant $C_{i}$ such that $\scal_{i}+|{\nabla f_{i}}|^{2}-2\lambda_{i}f_{i}=C_{i}$ holds on $\M_{i}$. At $p_{i}$, this implies
    \begin{equation}
        |{C_{i}}|\leq|{\scal_{i}}|(p_{i})+|{\nabla f_{i}}|^{2}(p_{i})+2|{\lambda_{i}f_{i}}|(p_{i})\leq n+A^{2}+2\lambda
         A=:C_{0},
    \end{equation}
    for all $i$. Therefore, as the argument in the proof of Lemma \ref{lm-uniform curvature bound in class m}, for any $R>0$, there is a uniform constant $\tilde{C}=\tilde{C}(n,A,C_{i},\lambda_{i},R)=\tilde{C}(n,A,C_{0},\lambda,R)$ such that 
    \begin{equation}
        |{f_{i}}|+|{\nabla f_{i}}|\leq \tilde{C}
    \end{equation} 
    on $B_{g_{i}}(p_{i},R)$ for all $i\gg1$. Then the Arzelà–Ascoli theorem provides a locally Lipschitz limit $f_{\infty}:X_{\infty}\rightarrow\R$ such that $f_{i}\rightarrow f_{\infty}$ locally uniformly.
    
    \vskip0.2cm
    Now we aim to show that $f_{\infty}$ is smooth on $\mathcal{R}$ and satisfies the soliton equation. For $q\in\mathcal{R}\cap B_{d_{\infty}}(p,R)$ and $\M_{i}\ni q_{i}\rightarrow q$, by Colding's volume convergence theorem  and Anderson's $\varepsilon$-regularity theorem \cite{C97,A90}, there exists constants $\Lambda,\alpha,r_{0}$ such that $r_{H}(\Lambda,1,\alpha)(q_{i})\geq r_{0}$ for $i\gg1$. Therefore, there exists a harmonic coordinate chart on $B_{g_{i}}(q_{i},r_{0})$ with $\mathcal{C}^{1,\alpha}_{H}(q_{i},r_{0})\leq\Lambda$. Since $q_{i}\rightarrow q$, this shows that $q_{i}\in B_{g_{i}}(p_{i},R+1)$ for $i\gg1$. As the proof of Lemma \ref{lm-uniform curvature bound in class m}, this implies a uniform upper bound of $\mathcal{C}^{2,\alpha}_{H}(q_{i},\frac{r_{0}}{8})$ bound for all $i\gg1$. Either apply the local Shi's estimate or the elliptic bootstrap argument of \eqref{eq-Bochner formula under harmonic coordinate}, one can conclude a uniform bound of $\mathcal{C}^{k,\alpha}_{H}(q_{i},\frac{r_{0}}{16})$ as $i\gg1$. This implies $B_{d_{\infty}}(q,\frac{r_{0}}{32})$ admits a smooth coordinate chart. By the soliton equation, this implies a uniform $\mathcal{C}^{k,\alpha}$ bound of $f_{i}$ on $B_{g_{i}}(q_{i},\frac{r_{0}}{16})$ for all $i\gg1$. Therefore, we obtain $f_{i}\rightarrow f_{\infty}$ smoothly on $B_{d_{\infty}}(q,\frac{r_{0}}{32})$. After passing to subsequence, we may assume $\lambda_{i}\rightarrow\lambda_{\infty}$ and this implies $\Ric_{g_{\infty}}+\Hess_{g_{\infty}}f_{\infty}=\lambda_{\infty}g_{\infty}$ on $\mathcal{R}$.
\end{proof}
Next, we explored whether $\lim_{r\rightarrow\infty}|{\Rm}|=0$ whenever $\liminf_{r\rightarrow\infty}r_{H}(\Lambda,1,\alpha)=\infty$. After imposing the decayed assumption on $\Ric$, we are able to confirm the following:
\begin{cl}\label{cl- scalar curvature with linear decayed}
    Let $(\M,g,f)$ be a complete gradient steady Ricci soliton. Suppose that there are constants $\Lambda,C>1$ and $\alpha\in(0,1)$ such that
    \begin{equation}
        |{\Ric}|(x)\leq\frac{C}{d_{g}(p,x)}\text{ and }r_{H}(\Lambda,1,\alpha)(x)\geq C^{-1}d_{g}(p,x)^{\frac{1}{3}},
    \end{equation}
    for all $x\in\M$ and some $p\in\M$. Then there is a constant $C'>0$ such that
    \begin{equation}\label{eq- cl3, 2}
        |{\Rm}|(x)\leq \frac{C'}{d_{g}(p,x)^{\frac{2}{3}}},       
    \end{equation}
    for all $x\in\M$. In particular, $\limsup_{r\rightarrow\infty}|{\Rm}|=0$.
\end{cl}
\begin{rmk}
    The authors are unaware whether the decay rate in \eqref{eq- cl3, 2} is optimal or not. In fact, one can follow the same strategy to show that for a steady gradient Ricci soliton, if $\limsup_{r\rightarrow\infty}|{\Ric}|=0$ and $\liminf_{r\rightarrow\infty}r_{H}(\Lambda,1,\alpha)=\infty$, then $\limsup_{r\rightarrow\infty}|{\Rm}|=0$.
\end{rmk}
\begin{proof}[Proof of Corollary \ref{cl- scalar curvature with linear decayed}]
By Corollary \ref{cl-bounded curvature under uniform lower bound of r(1,alpha)}, $\Rm$ is bounded on $\M$. Hence, by the local Shi's estimate (see \cite[Remark 2.7]{D17}), we yield $|{\nabla\Ric}|\leq C'd_{g}(p,\cdot)^{-1}$ for some constant $C'>0$. For $q\in B_{g}(p,R)\setminus B_{g}(p,\frac{R}{2})$, we rescaled the metric by $g_{R}:=R^{-\frac{2}{3}}g$ on $B_{g}(q,C^{-1}(R/2)^{\frac{1}{3}})$. Since $r_{H}(\Lambda,1,\alpha)(q)\geq C^{-1}(R/2)^{\frac{1}{3}}$, this implies $r_{H,g_{R}}(\Lambda,1,\alpha)(q)\geq 1/C2^{\frac{1}{3}}$ and exists a harmonic coordinate chart $(B_{g_{R}}(q,1/(C2^{\frac{1}{3}})),x^{1},\cdots,x^{n})$ with $x^{i}(q)=0$ with $i=1,\cdots,n$ and $\mathcal{C}^{1,\alpha}_{H,g_{R}}(q,1/C2^{\frac{1}{3}})\leq\Lambda$. Since the harmonic coordinate guarantees
\begin{equation}
    \Delta g_{ij}=-2\Ric_{ij}+Q(g,\partial g)_{ij},
\end{equation}
for some quadratic term $Q$, and for the rescaled metric, we have 
\begin{equation}
    |{\Ric_{g_{R}}}|_{g_{R}}\leq \frac{2C}{R^{\frac{1}{3}}}\ll1\text{ and }|{\nabla\Ric}_{g_{R}}|_{g_{R}}\leq C',
\end{equation}
on $B_{g_{R}}\left(q,\frac{1}{C2^{\frac{1}{3}}}\right)$. Therefore, we have the uniform $\mathcal{C}^{\alpha}$ bound on R.H.S. on $B_{g_{R}}(q,\frac{1}{C2^{\frac{1}{3}}})$. Hence, by the elliptic Schauder estimate, we obtain a uniform upper bound of $\mathcal{C}^{2,\alpha}_{H,g_{R}}(q,\frac{1}{2C2^{\frac{1}{3}}})$ on this harmonic coordinate and $|{\Rm_{g_{R}}}|(q)\leq C''$ for some constant $C''(\Lambda,C,C',n)>0$, which implies $|{\Rm_{g}}|(q)\leq C''/R^{2/3}\leq2C''/d_{g}(p,q)^{2/3}$ for $d_{g}(p,q)\gg1$. This completes the proof.
\end{proof}
\begin{rmk}
The main technical difficulty in working with harmonic coordinates arises when applying local Schauder or $\mathcal{L}^{p}$ estimates to control the $\mathcal{C}^{k,\alpha}_{H}$ norm of $f$. Such estimates typically require an a priori $\mathcal{C}^{0}$ bound on $f$. However, this presents a difficulty in our setting, since the potential function $f$ is only defined up to an additive constant and is invariant under the natural scaling of the soliton. As a result, no uniform $\mathcal{L}^{\infty}$ bound for $f$ is available a priori.
\end{rmk}

\section{Proof of Theorem \ref{thm-bounded Rm with L1 condition}}\label{proof of main theorem}

As a byproduct of Lemma \ref{lm-uniform curvature bound in class m}, we now present the proof of Theorem \ref{thm-bounded Rm with L1 condition}.
\begin{proof}[Proof of Theorem \ref{thm-bounded Rm with L1 condition}]
    By Lemma \ref{lm-uniform curvature bound in class m}, it suffices to prove that 
    
    \noindent $(\M,g,f-f(p),p)\in\mathfrak{M}(n,K,1,r_{0},0,\alpha,\Lambda,C_{0})$ for some $K,r_{0},\alpha,\Lambda,C_{0}$ depending on $g,f$. Since $|{\Ric}|\rightarrow0$ as $r\rightarrow\infty$, $|{\nabla f}|^{2}+\scal=1$, and $f$ is proper, we may assume there is a $R_{0}$ such that $|{\nabla f}|>1/2$ on $\{F:=-f\geq R_{0}\}$. Consider the flow $\varphi_{t}$ such that $\frac{d}{dt}\varphi_{t}=\nabla F/|{\nabla F}|^{2}$ with $\varphi_{0}=\text{id}$ on $\{F\geq R_{0}\}$. Similar to the argument as in \cite[Section 5]{CH25}, by \ref{Ass-1}, on $\Sigma_{R_{0}}:=\{F=R_{0}\}$, 
    \begin{equation}
        \left|\frac{d}{dt}\varphi_{R_{0}-t}^{*}g|_{F=t}\right|=\left|\frac{2\Ric(g)|_{\{F=t\}}}{|{\nabla F}|^{2}}\right|\leq Ch\left(t-C\right),
    \end{equation}
    for $t\gg R_{0}$, where the last inequality follows from Lemma \ref{lm-linear equivalent of F} and $h$ is non-increasing. Since $h\in\mathcal{L}^{1}$ and is non-increasing when $t\gg1$, there is a $\mathcal{C}^{0}$-metric $g_{\infty}$ on $\Sigma_{R_{0}}$ such that $\varphi_{t-R_{0}}^{*}g|_{\Sigma_{t}}$ converges to $g_{\infty}$ uniformly as $t\rightarrow\infty$.

    \vskip0.2cm
    Now, on $\Sigma_{a,b}:=\{a<F<b\}\simeq\Sigma_{R_{0}}\times(a,b)$ with $a>R_{0}$, we denote the map $\Psi(x):=(\varphi_{F(x)-R_{0}}(x),F(x))$, which is the desire diffeomorphism between $\Sigma_{a,b}$ and $\Sigma_{R_{0}}\times(a,b)$. Then
    \begin{equation}\label{eq-C0 closedness of metric}
        \|(\Psi^{-1})^{*}g-(dt^{2}+g_{\infty})\|_{dt^{2}+g_{\infty}}\leq C\left(1-\frac{1}{|{\nabla F}|}\right)+\sup_{t\in(a,b)}\|\varphi_{t-R_{0}}^{*}g|_{\Sigma_{t}}-g_{\infty}\|_{g|_{\Sigma_{t}}}\leq C\int_{a-C}^{\infty}h(s)ds,
    \end{equation}
    as $a\rightarrow\infty$ by Lemma \ref{lm-linear equivalent of F}. Therefore, for all $p\in\Sigma_{R_{0}}$, we have $(\overline{\Sigma_{t,t+4}},\Psi^{*}g,\varphi_{t+2-R_{0}}(p))$ converges to $(\Sigma_{R_{0}}\times[0,4],dt^{2}+g_{\infty},(p,2))$ as $t\rightarrow\infty$ in the $\mathcal{C}^{0}$-sense. 

    \vskip0.2cm
    Next, fixed $\epsilon>0$. We aim to show that there is a uniform radius $r_{\epsilon}>0$ such that $\vol_{g}(x,r_{\epsilon})/\omega_{n}r^{n}_{\epsilon}\geq1-\epsilon$ for all $r(x)\gg1$. This can be achieved by the $\mathcal{C}^{0}$-convergence via the following argument: Take $\delta>0$ to be determined later, then there is a $R_{\delta}>0$ such that for $\ell,s>R_{\delta}$,
    \begin{equation}\label{eq-C) equivalence between different slices}
        (1-\delta)\Psi_{*}g|_{\Sigma_{s,s+4}}\leq \Psi_{*}g|_{\Sigma_{\ell,\ell+4}}\leq(1+\delta)\Psi_{*}g|_{\Sigma_{s,s+4}},
    \end{equation}
    by (\ref{eq-C0 closedness of metric}). Since $\Sigma_{R_{{0}}}$ is compact, for a fixed $\ell>R_{\delta}$, there is a constant $r_{\delta}>0$ such that 
    \begin{equation}
        \frac{\vol_{\Psi_{*}g|_{\Sigma_{\ell,\ell+4}}}\left((x,\ell+2),r\right)}{\omega_{n}r^{n}}\geq1-\delta,,
    \end{equation}
    for all $x\in\Sigma_{R_{0}}$ and $r\in(0,r_{\delta})$. By (\ref{eq-C) equivalence between different slices}), this implies for $s>R_{\delta}$, 
    \begin{equation}
        \frac{\vol_{\Psi_{*}g|_{\Sigma_{s,s+4}}}\left((x,s+2),r\right)}{\omega_{n}r^{n}}\geq\frac{\vol_{(1+\delta)^{-1}\Psi_{*}g|_{\Sigma_{\ell,\ell+4}}}\left((x,s+2),(1-\delta)^{\frac{1}{2}}r\right)}{\omega_{n}r^{n}}\geq(1-\delta)^{\frac{n}{2}+1}(1+\delta)^{-\frac{n}{2}},
    \end{equation}
    for all $r\in(0,r_{\delta})$ and $x\in\Sigma_{R_{0}}$. Choose $\delta$ small enough so that $(1-\delta)^{\frac{n}{2}+1}(1+\delta)^{-\frac{n}{2}}\geq1-\epsilon$, then we confirm for $r(x)>R_{\epsilon}$, the almost Euclidean volume ratio property $\vol_{g}(x,r_{\epsilon})/\omega_{n}r^{n}_{\epsilon}\geq1-\epsilon$ holds for all $r\in(0,r_{\epsilon})$. Since $\M\setminus\Psi^{-1}(\Sigma_{R_{\delta}}\times(R_{\delta},\infty))$ is compact, we may assume the almost volume ratio property holds for all $x\in\M$ and $r\in(0,r_{\epsilon})$ (possibly adjusting $r_{\epsilon}$). By \ref{Ass-1}, the Ricci curvature is bounded on $\M$. Hence, we yield a uniform lower bound for $r_{H}(\Lambda,1,\alpha)(x)$ for all $\Lambda>1$, $\alpha\in(0,1)$, and $x\in\M$. See \cite[Remark 3.3]{A90} or \cite[Theorem 2.1]{H19} for instance. By Lemma \ref{lm-uniform curvature bound in class m}, this implies that $\Rm$ is uniformly bounded on $\M$. Since $|{\Ric}|\rightarrow0$ as $r\rightarrow\infty$ and $r_{H}(\Lambda,1,\alpha)\geq r_{0}$ for some $r_{0}>0$, the limiting metric $dt^{2}+g_{\infty}$ is Ricci flat, hence $g_{\infty}$ is Ricci flat and smooth on $\Sigma_{R_{0}}$. The smooth convergence follows from Shi's estimate and the unique $\mathcal{C}^{0}$-convergence.

    \vskip0.2cm
    Define $H(t):=\int_{t}^{\infty}h(s)ds$. It remains to prove the convergence rate is of $\BigO(H(r))$ with any $\mathcal{C}^{k}$ norm. Since $\Rm$ is bounded, by local Shi's estimate \cite[Remark 2.7]{D17}, for any $k\in\mathbb{N}_{0}$, there is a constant $C_{k}>0$ such that
    \begin{equation}\label{eq-proof of thm1- nabla k Ric}
        |{\nabla^{k}\Ric}|\leq C_{k}h(r(x)),
    \end{equation}
    for all $x\in\M$. Also, the connection of $g|_{\Sigma_{R}}$ is $\nabla_{g_{R}}X=\nabla X-\left(\nabla X\cdot\frac{\nabla F}{|{\nabla F}|}\right)\frac{\nabla F}{|{\nabla F}|}$ for vector filed $X$ and $\nabla_{g|_{\Sigma_{R}}}\omega=\nabla\omega+\omega(\frac{\nabla F}{|{\nabla F}|})\nabla\cdot\frac{\nabla F}{|{\nabla F}|}$ for any 1-form $\omega$. Now, we define the pullback metric  $g_{t}:=\varphi^{*}_{t-R_{0}}g|_{\Sigma_{t}}$ on $\Sigma_{R_{0}}$ for $t\gg R_{0}$. Then for any tensor $T$ on $\Sigma_{R_{0}}$,
    \begin{equation}\label{eq-proof of thm1 - difference of connection}
        \nabla_{g_{t_{2}}}T-\nabla_{g_{t_{1}}}T=\int_{t_{1}}^{t_{2}}\nabla_{g_{\tau}}\left.\frac{\Ric(\varphi^{*}_{\tau-R_{0}}g)}{|{\nabla f}|^{2}\circ\varphi_{\tau-R_{0}}}\right|_{\Sigma_{R_{0}}}\ast Td\tau = \BigO(H(t_{1}))|T|_{g_{\infty}},
    \end{equation}
    for all $t_{1}\leq t_{2}\leq \infty$. Take $T=g_{t_{1}}$, then we obtain
    \begin{equation}\label{eq-proof of thm 1 - k=1}
        \nabla_{g_{t_{2}}}g_{t_{1}}=\BigO(H(\min\{t_{1},t_{2}\})),
    \end{equation}
    for all $t_{1},t_{2}\gg1$. We claim that $\nabla_{g_{t_{1}}}^{j}\nabla^{k-j}_{g_{t_{2}}}g_{t_{1}}=\BigO(H(\min\{t_{1},t_{2}\}))$ for all $0\leq j\leq k$ and prove it by induction on $k$. $k=1$ follows from \eqref{eq-proof of thm 1 - k=1}. For $k>1$, substitute $T=\nabla_{g_{t_{2}}}^{k-1}g_{t_{1}}$ into \eqref{eq-proof of thm1 - difference of connection}, 
    \begin{equation}
        \nabla^{k}_{g_{t_{2}}}g_{t_{1}}=\int_{t_{1}}^{t_{2}}\nabla_{g_{\tau}}\left.\frac{\Ric(\varphi^{*}_{\tau-R_{0}}g)}{|{\nabla f}|^{2}\circ\varphi_{\tau-R_{0}}}\right|_{\Sigma_{R_{0}}}\ast \nabla_{g_{t_{2}}}^{k-1}g_{t_{1}}d\tau +\nabla_{g_{t_{1}}}\nabla_{g_{t_{2}}}^{k-1}g_{t_{1}}=\BigO(H(t))+\nabla_{g_{t_{1}}}\nabla_{g_{t_{2}}}^{k-1}g_{t_{1}},
    \end{equation}
    by the induction hypothesis. Now, by \eqref{eq-proof of thm1 - difference of connection} again, 
    \begin{equation}
        \begin{aligned}
            \nabla_{g_{t_{1}}}\nabla_{g_{t_{2}}}^{k-1}g_{t_{1}}-\nabla_{g_{t_{1}}}^{2}\nabla_{g_{t_{2}}}^{k-2}g_{t_{1}}=&~\nabla_{g_{t_{1}}}\int_{t_{1}}^{t_{2}}\nabla_{g_{\tau}}\left.\frac{\Ric(\varphi^{*}_{\tau-R_{0}}g)}{|{\nabla f}|^{2}\circ\varphi_{\tau-R_{0}}}\right|_{\Sigma_{R_{0}}}\ast \nabla_{g_{t_{2}}}^{k-2}g_{t_{1}}d\tau \\
            =&~\int_{t_{1}}^{t_{2}}\nabla_{g_{t_{1}}}\nabla_{g_{\tau}}\left.\frac{\Ric(\varphi^{*}_{\tau-R_{0}}g)}{|{\nabla f}|^{2}\circ\varphi_{\tau-R_{0}}}\right|_{\Sigma_{R_{0}}}\ast \nabla_{g_{t_{2}}}^{k-2}g_{t_{1}}d\tau \\
            &~+\int_{t_{1}}^{t_{2}}\nabla_{g_{\tau}}\left.\frac{\Ric(\varphi^{*}_{\tau-R_{0}}g)}{|{\nabla f}|^{2}\circ\varphi_{\tau-R_{0}}}\right|_{\Sigma_{R_{0}}}\ast \nabla_{g_{t_{1}}}\nabla_{g_{t_{2}}}^{k-2}g_{t_{1}}d\tau\\
            =:&~A+B.
        \end{aligned}
    \end{equation}
    For $B$, substitute $T=\nabla^{k-2}_{g_{t_{2}}}g_{t_{1}}$ into \eqref{eq-proof of thm1 - difference of connection}, we get $\nabla_{g_{t_{1}}}\nabla^{k-2}_{g_{t_{2}}}g_{t_{1}}=\BigO(H(\min\{t_{1},t_{2}\}))$ by induction hypothesis. Therefore, $B=\BigO(H(\min\{t_{1},t_{2}\}))$ for $t_{1},t_{2}\gg1$. For $A$, substitute $t_{2}=\tau$ and $T=\nabla_{g_{\tau}}\left.\frac{\Ric(\varphi^{*}_{\tau-R_{0}}g)}{|{\nabla f}|^{2}\circ\varphi_{\tau-R_{0}}}\right|_{\Sigma_{R_{0}}}$ into \eqref{eq-proof of thm1 - difference of connection}, 
    \begin{equation}\label{eq-proof of Theorem 1- nabla nabla Ric}
    \begin{aligned}
        \nabla_{g_{t_{1}}}\nabla_{g_{\tau}}\left.\frac{\Ric(\varphi^{*}_{\tau-R_{0}}g)}{|{\nabla f}|^{2}\circ\varphi_{\tau-R_{0}}}\right|_{\Sigma_{R_{0}}}=&~\nabla_{g_{\tau}}\nabla_{g_{\tau}}\left.\frac{\Ric(\varphi^{*}_{\tau-R_{0}}g)}{|{\nabla f}|^{2}\circ\varphi_{\tau-R_{0}}}\right|_{\Sigma_{R_{0}}}\\
        &~+\int_{\tau}^{t_{1}}\nabla_{g_{s}}\left.\frac{\Ric(\varphi^{*}_{s-R_{0}}g)}{|{\nabla f}|^{2}\circ\varphi_{s-R_{0}}}\right|_{\Sigma_{R_{0}}}\ast\nabla_{g_{\tau}}\left.\frac{\Ric(\varphi^{*}_{\tau-R_{0}}g)}{|{\nabla f}|^{2}\circ\varphi_{\tau-R_{0}}}\right|_{\Sigma_{R_{0}}}ds\\
        =&~\BigO(h(\tau)),
    \end{aligned}
    \end{equation}
    where we use \eqref{eq-proof of thm1- nabla k Ric} to estimate first term. This implies $A$ is also of $\BigO(H(\min\{t_{1},t_{2}\})$ by induction hypothesis. Continue this process, and one can show that
    \begin{equation}\label{eq-proof of thm2- nabla j nabla k-j}
        \nabla^{j}_{g_{t_{1}}}\nabla_{g_{t_{2}}}^{k-j}g_{t_{1}}=\sum_{\ell=0}^{k-j-2}\nabla_{g_{t_{1}}}^{j+\ell}\int_{\tau}^{t_{1}}\nabla_{g_{s}}\left.\frac{\Ric(\varphi^{*}_{s-R_{0}}g)}{|{\nabla f}|^{2}\circ\varphi_{s-R_{0}}}\right|_{\Sigma_{R_{0}}}\ast\nabla_{g_{t_{2}}}^{k-j-1-\ell}g_{t_{1}}ds,
    \end{equation}
    and
    
    \begin{equation}\label{eq-proof of thm2- nabla k Ric}
        \begin{aligned}
            &~\nabla_{g_{t_{1}}}^{m}\nabla_{g_{\tau}}\left.\frac{\Ric(\varphi^{*}_{\tau-R_{0}}g)}{|{\nabla f}|^{2}\circ\varphi_{\tau-R_{0}}}\right|_{\Sigma_{R_{0}}}\\
            =&~\sum_{\ell=0}^{m-1}\nabla_{g_{t_{1}}}^{m-1-\ell}\int_{\tau}^{t_{1}}\nabla_{g_{s}}\left.\frac{\Ric(\varphi^{*}_{s-R_{0}}g)}{|{\nabla f}|^{2}\circ\varphi_{s-R_{0}}}\right|_{\Sigma_{R_{0}}}\ast\nabla_{g_{\tau}}^{\ell+1}\left.\frac{\Ric(\varphi^{*}_{\tau-R_{0}}g)}{|{\nabla f}|^{2}\circ\varphi_{\tau-R_{0}}}\right|_{\Sigma_{R_{0}}}ds,
        \end{aligned}
    \end{equation}
    for all $m\geq0$.
    Applying the inductive argument on \eqref{eq-proof of thm2- nabla k Ric}, by \eqref{eq-proof of thm1- nabla k Ric}, one can show that 
    \begin{equation}
        \nabla_{g_{t_{1}}}^{m}\nabla_{g_{\tau}}\left.\frac{\Ric(\varphi^{*}_{\tau-R_{0}}g)}{|{\nabla f}|^{2}\circ\varphi_{\tau-R_{0}}}\right|_{\Sigma_{R_{0}}}=\BigO(h(\tau)),
    \end{equation}
    with the same argument as \eqref{eq-proof of Theorem 1- nabla nabla Ric}. Therefore, we conclude that $\nabla^{j}_{g_{t_{1}}}\nabla_{g_{t_{2}}}^{k-j}g_{t_{1}}nabla^{k}_{g_{t_{2}}}g_{t_{1}}=\BigO(H(\min\{t_{1},t_{2}\}))$ by \eqref{eq-proof of thm2- nabla j nabla k-j}. Choose $t_{2}=\infty$, this implies
    \begin{equation}
        |g_{t_{1}}-g_{\infty}|_{\mathcal{C}^{k}_{g_{\infty}}}=\BigO(H(t_{1})),
    \end{equation}
    as $t_{1}\gg1$. On the other hands,
     since 
     \begin{equation}
         \nabla_{\overline{g_{\infty}}}\Psi_{*}g=\Psi_{*}\left[\left(\Hess F\right)\ast|{\nabla F}|^{-3}\right] +\frac{2}{|{\nabla F}|^{2}\circ\Psi^{-1}}\Psi_{*}\left(\Ric(g)|_{g_{\Sigma}}\right)+\nabla_{g_{\infty}}g_{t}=\BigO(H(t)),
     \end{equation}
     on $\R^{+}\times\N$, one can repeat the estimate as above to show that for any $k\in\mathbb{N}_{0}$,
     \begin{equation}
         |\nabla_{\overline{g_{\infty}}}^{k}((\Psi^{-1})^{*}g-\overline{g_{\infty}}))|_{\overline{g_{\infty}}}=|\nabla_{\overline{g_{\infty}}}^{k}(\Psi^{-1})^{*}g|_{\overline{g_{\infty}}}=\BigO(H(t)),
     \end{equation}
     as $t\rightarrow\infty$ by \eqref{eq-proof of thm1- nabla k Ric}. This completes the proof.
\end{proof}
\begin{proof}[Proof of Corollary \ref{cl-curvature estimate with h(r)=exp(-r)}]
    This is a direct consequence of Theorem \ref{thm-bounded Rm with L1 condition}.
\end{proof}
\begin{proof}[Proof of Corollary \ref{cl-curvature estimate with Dini H}]
    By Theorem \ref{thm-bounded Rm with L1 condition} and the assumption that $\lim_{r\rightarrow\infty}|{\Rm}|=0$, there exists a \textit{closed flat manifold $(\N^{n-1},g_{\infty})$} such that
    \begin{equation}
        \|(\Psi^{-1})^{*}g-\overline{g_{\infty}}\|_{\mathcal{C}^{k}(\overline{g_{\infty}})}\leq C_{k}\int_{r(x)}^{\infty}h(s)ds=:C_{k}\frac{H(r(x))}{r(x)},
    \end{equation}
    when $r(x)\gg1$. Since $\overline{g_{\infty}}$ is a flat manifold, this shows that $|{\Rm}|=\BigO(H(r)/r)$ as $r\rightarrow\infty$ by Lemma \ref{lm-linear equivalent of F}. Therefore, 
    \begin{equation}
        \lim_{r\rightarrow\infty}r|{\Rm}|\leq\lim_{r\rightarrow\infty}H(r)=0,
    \end{equation}
    by the assumption of $H$. Hence, we get $|{\Rm}|\leq Ce^{-r(x)}$ for some constant $C>0$ by \cite[Theorem 2]{C19}. The assertion follows from applying Corollary \ref{cl-curvature estimate with h(r)=exp(-r)}.
\end{proof}


\begin{thebibliography}{100}
\bibitem[A90]{A90}M.T. Anderson. \textit{Convergence and rigidity of manifolds under Ricci curvature bounds}. Invent Math \textbf{102}, 429–445 (1990). 

\bibitem[Ba20a]{Ba20a}R. Bamler, \textit{Entropy and heat kernel bounds on a Ricci flow background}, 
arXiv:2009.11754.

\bibitem[Ba20b]{Ba20b}R. Bamler, \textit{Structure theory of non-collapsed limits of Ricci flows}, arXiv preprint, 2020.


\bibitem[Ba23]{Ba23}R. Bamler, \textit{Compactness theory of the space of Super Ricci flows}, 
Invent. Math. \textbf{233} (2023), 1121–1277.

\bibitem[B20]{B20}S. Brendle, \textit{Ancient solutions to the Ricci flow in dimension 3}, 
Acta Math. \textbf{225} (2020), no. 1, 1–102.

\bibitem[CH24]{CH24}X. Cao and H. Tran, \textit{Geometry and analysis of gradient Ricci solitons in dimension four}, 
in: Surveys in Differential Geometry, Vol. 27, no. 1, 2022, pp. 213–233.

\bibitem[C19]{C19}P.-Y. Chan, \textit{Curvature estimates for steady Ricci solitons
}, Trans. Am. Math. Soc., \textbf{371}, no. 12, 8985-9008 (2019).

\bibitem[CH25]{CH25}P.-Y. Chan and M. Hsiao, \textit{Curvature estimates for steady and expanding solitons in higher dimensions}, arXiv:2512.05625 (2025).

\bibitem[CMZ22]{CMZ22}
P.-Y. Chan, Z. Ma and Y. Zhang, \textit{Volume growth estimates of gradient Ricci solitons}, 
J. Geom. Anal. \textbf{32} (2022), Paper No. 291, 36 pp.

\bibitem[CN15]{CN15}J. Cheeger and A. Naber, \textit{Regularity of Einstein manifolds and the codimension 4 conjecture}, 
Ann. of Math. (2) \textbf{182} (2015), no. 3, 1093–1165.

\bibitem[C09]{C09}
B.-L. Chen, \textit{Strong uniqueness of the Ricci flow}, 
J. Differential Geom. \textbf{82} (2009), 363–382.

\bibitem[C20]{C20}C.-W. Chen. \textit{Shi-type estimates of the Ricci flow based on Ricci curvature}. Ann. Sc. Norm. Super. Pisa Cl. Sci. \textbf{20}, 1553–1580 (2020).

\bibitem[C97]{C97}T. H. Colding, \textit{Ricci curvature and volume convergence}, 
Ann. of Math. (2) \textbf{145} (1997), no. 3, 477–501.

\bibitem[D12]{D12}A. Deruelle. \textit{Steady gradient Ricci soliton with curvature in L1}. Commun. Anal. Geom. \textbf{20}, no. 1, 31-53 (2012).

\bibitem[D17]{D17}
A. Deruelle, \textit{Asymptotic estimates and compactness of expanding gradient Ricci solitons}, 
Ann. Sc. Norm. Super. Pisa Cl. Sci. (5) \textbf{17} (2017), 485–530.

\bibitem[H82]{H82}R. S. Hamilton, \textit{Three-manifolds with positive Ricci curvature}, 
J. Differential Geom. \textbf{17} (1982), no. 2, 255–306.

\bibitem[H95]{H95}
R. S. Hamilton, \textit{The formation of singularities in the Ricci flow}, 
Surveys in Differential Geometry (Cambridge, MA, 1993), \textbf{2}, 7-136, International Press, Cambridge, MA, 1995.

\bibitem[HM11]{HM11}R. Haslhofer and R. Müller, \textit{A compactness theorem for complete Ricci shrinkers}, 
Geom. Funct. Anal. \textbf{21} (2011), 1091–1116.


\bibitem[H19]{H19}
S. Huang, \textit{A note on existence of exhaustion functions and its applications}, 
J. Geom. Anal. \textbf{29} (2019), no. 2, 1649-1659.

\bibitem[LW24]{LW24}Y. Li and B. Wang, \textit{Heat kernel on Ricci shrinkers (II)}, 
Acta Math. Sci. \textbf{44} (2024), 1639–1695.

\bibitem[MW11]{MW11}O. Munteanu and J. Wang, \textit{Smooth metric measure spaces with non-negative curvature}, 
Commun. Anal. Geom. \textbf{19} (2011), no. 3, 451–486.

\bibitem[MSW19]{MSW19}O. Munteanu, C.-J. A. Sung, and J. Wang, \textit{Poisson equation on complete manifolds}, Adv. Math. \textbf{348} (2019), 81–145.

\bibitem[P02]{P02}G. Perelman, \textit{The entropy formula for the Ricci flow and its geometric applications}, arXiv preprint, 2002.

\bibitem[W25]{W25}L. Wu, \textit{Curvature estimates for gradient steady Ricci solitons}, 
Period. Math. Hungar. \textbf{92} (2026), 46–56.

\bibitem[Z06]{Z06}X. Zhang, \textit{Compactness theorems for gradient Ricci solitons}, 
J. Geom. Phys. \textbf{56} (2006), no. 12, 2481–2499.
\end{thebibliography}
\end{document}